\documentclass[12pt,reqno]{amsart}

\usepackage[arrow,matrix,curve]{xy}

\usepackage[dvips]{graphicx} 

\usepackage{amssymb, latexsym, amsmath, amscd, array, hyperref
}

\newtheorem{theorem}{Theorem}[section]
\newtheorem{lemma}[theorem]{Lemma}

\theoremstyle{definition}
\newtheorem{definition}[theorem]{Definition}
\newtheorem{example}[theorem]{Example}
\numberwithin{equation}{section}
\newcommand\R {{\mathbb R}}

\newcommand\CR{\mathrm{CR}}

\newcommand\LCD{\mathrm{LCD}}

\author{Mikhail G. Katz}\address{M. Katz, Department of Mathematics,
Bar Ilan University, Ramat Gan 52900 Israel}
\email{katzmik@macs.biu.ac.il}

\begin{document}

\thispagestyle{empty}

%\huge

\title{Convexity, critical points, and connectivity radius}

\begin{abstract}
We study the level sets of the distance function from a boundary point
of a convex set in Euclidean space.  We provide a lower bound for the
range of connectivity of the level sets, in terms of the critical
points of the distance function in the sense of
Grove--Shiohama--Gromov--Cheeger.
\end{abstract}

\maketitle
%\tableofcontents

\section{Introduction}

Critical point theory for non-smooth functions such as the distance
function from a point in a Riemannian manifold was developed by Grove
and Shiohama~\cite{GS}, Gromov \cite{Gr81}, and Cheeger \cite{Ch91}
(see also \cite{Ka89}).

We will exploit this notion of criticality to derive an optimal lower
bound for the connectivity radius at a boundary point of a convex set
in Euclidean space in terms of the least critical distance.  The
connectedness of the level sets is a question posed at MO.%
\footnote{See \url{https://mathoverflow.net/q/227562}}
We apply the techniques along the lines of
Grove--Shiohama--Gromov--Cheeger, to develop a kind of Morse theory to
provide an optimal answer.  Here the answer is optimal in the sense
that it is easy to give nontrivial cases where the bound is optimal.

\section{Connectivity radius and critical distance}

Let~$K\subseteq\R^n$ be a closed convex set of dimension at least~$2$.
Consider a boundary point~$O\in\partial K$.  Let~$S_r(O)\subseteq\R^n$
be the sphere of radius~$r>0$ centered at~$O$.

\begin{definition}
The \emph{connectivity radius}~$\CR(O)$ of~$O\in\partial K$ is the
supremum of all~$\epsilon$ such that for all~$r<\epsilon$ the
intersection~$S_r(O)\cap K$ is connected.
\end{definition}

\begin{definition}
A point~$P\in \partial K\setminus\{O\}$ is \emph{$O$-critical}, or
\emph{critical} for short, if one of the following equivalent
conditions is satisfied:
\begin{enumerate}
\item
$\langle O-P, X-P\rangle \geq 0$ for all~$X\in K$;
\item
the affine hyperplane through~$P$ orthogonal to~$OP$ is a supporting
hyperplane for~$K$.
\end{enumerate}
\end{definition}

\begin{example} 
If~$K\subseteq\R^2$ is a disk and~$O\in\partial K$, then the
only~$O$-critical point is the antipodal point of~$O$.
\end{example}

\begin{definition}
The least critical distance~$\LCD(O)$ is the infimum of~$|OP|$ where
the infimum is taken over all~$O$-critical points~$P\in \partial K$.
\end{definition}

\section{The results}

\begin{theorem}
\label{t17}
For each~$O\in\partial K$ we have~$\CR(O)\geq \LCD(O)$.
\end{theorem}

\begin{example}
Let~$K\subseteq\R^2$ be an acute-angled triangle.  If~$O$ a vertex of
$K$, then~$\CR(O)$ is the length of the altitude from~$O$ to the
opposite side, and the foot of the altitude is an~$O$-critical point.
If~$O$ contained in an open side of~$K$, then~$\CR(O)$ is the smaller
of the two distances from~$O$ to the remaining two sides, and the foot
of each perpendicular is a critical point.
\end{example}

The lower bound provided by the theorem is nontrivial due to the
following lemma.

\begin{lemma}
Let~$K\subseteq\R^n$ be a convex set, and let~$O\in\partial K$.  Then
$\LCD(O)>0$.
\end{lemma}

\begin{proof}
If~$P_1$ and~$P_2$ are critical points with angle~$\measuredangle
P_1OP_2\leq\frac{\pi}{3}$ then by the Pythagorean Theorem we
have~$\frac{|OP_1|}{|OP_2|}\leq2$.  

Suppose~$(P_i)$ is a sequence of critical points tending to~$O$.
Passing to a subsequence if necessary, we can assume
that~$\frac{|OP_i|}{|OP_{i+1}|}>2$ for each~$i$.
Then~$\measuredangle{}P_i OP_j>\frac{\pi}{3}$ whenever~$i<j$.  But the
sphere of directions at~$O$ can only contain finitely many directions
such that all pairwise angles are greater than~$\frac{\pi}{3}$.  The
contradiction shows that the critical points must be bounded away
from~$O$.
\end{proof}

We first illustrate the theorem by considering the case~$n=2$.

\begin{lemma}
\label{l18}
Let~$K\subseteq\R^2$ be compact and convex.  Let~$O\in\partial K$ and
suppose~$S_r(O)\cap K$ has more than one connected component.  Then
there is an~$O$-critical point~$C$ with~$|OC|\leq r$.
\end{lemma}

\begin{proof}
Let~$B_r(O)$ be the closed disk of radius~$r$.  The hypothesis of the
lemma implies that the curve~$\partial K\cap B_r(O)$ also has more
than one connected component.  Let~$\gamma\subseteq\partial K\cap
B_r(O)$ be a connected component not containing the point~$O$.  If
$\gamma$ is a single point then it is $O$-critical.  Otherwise,
let~$A,B\in \gamma$ be the endpoints of the curve~$\gamma$.
Clearly~$|AO|=|BO|=r$.  If~$\gamma\subseteq S_r(O)$ then each interior
point of $\gamma$ is $O$-critical, proving the bound in this case.
Thus we may assume that some points of $\gamma$ lie in the interior
of~$B_r(O)$.  Let~$C\in\gamma$ be the point at least distance~$|CO|$
from~$O$.  By convexity of~$K$ and first variation, the line
through~$C$ orthogonal to~$OC$ is a supporting line for~$K$.  Thus~$C$
is an~$O$-critical point and~$|OC|<r$ in this case.
\end{proof}

At every non-critical point, there is a tangent vector with positive
(outward) $O$-radial component and pointing toward the interior
of~$K$.  A standard partition of unity argument using the convexity of
the tangent cone of $K$ at boundary points yields the following lemma.

\begin{lemma}
\label{l19}
Let $K\subseteq\R^n$ be convex.  Let~$L= S_r(O) \cap K$ be a level set
not containing any critical points.  Then there exists a smooth vector
field along~$L$ with constant positive radial component and pointing
toward the interior of~$K$.
\end{lemma}

We will exploit Lemma~\ref{l19} to prove our main theorem.

\begin{proof}[Proof of Theorem~$\ref{t17}$]
The case when~$K$ is 2-dimensional was treated in Lemma~\ref{l18}.  We
now treat the general case $n\geq2$.  Let 
\[
K_r = S_r(O)\cap K.
\]
Suppose~$K_r$ is not connected.  We will show that $r>\LCD(O)$.

Consider a pair of distinct connected components~$X,Y$ of $K_r$.  We
identify $O$ with the origin and choose rays~$\R^+ x$,~$\R^+ y$
meeting~$X$ and~$Y$ respectively.  Let~$\beta\leq r$ be the infimum of
radii~$s\leq r$ such that the connected components of the
points~$x_{s}=\R^+ x\,\cap\, K_{s}$ and~$y_{s}=\R^+ y\,\cap \, K_{s}$
are still distinct.  Since~$\dim K\geq2$, a path in $K$ connecting the
two rays can be chosen to avoid the point~$O$ and then pulled in
radially to the level containing a point of the path nearest $O$.
Hence~$\beta>0$.

Let us show that points $x_\beta$ and $y_\beta$ are in the same
connected component of $K_\beta$.  Suppose otherwise.  Then there are
disjoint open sets $U,V\subseteq\R^n$ such that $x_\beta\in U$,
$y_\beta\in V$, and $K_\beta\subseteq U\cup V$.  Let
$\epsilon_n=\frac{1}{n}$.  Since $K$ is star-shaped at $O$ and closed,
the intersections $K_{\beta-\epsilon_n}\cap U$ and
$K_{\beta-\epsilon_n}\cap V$ are still non-empty for sufficiently
large $n$.  By definition of $\beta$, the points
$x_{\beta-\epsilon_n}$ and $y_{\beta-\epsilon_n}$ are in the same
connected component of $K_{\beta-\epsilon_n}$.  Therefore there exists
a point $z_n\in K_{\beta-\epsilon_n}$ such that $z_n\not\in U\cup V$.
Passing to a subsequence if necessary we can assume that $(z_n)$
converges.  Let $z=\lim_{n\to\infty} z_n$.  By compactness of $K$, we
have $z\in K_\beta$.  On the other hand by construction $z\not\in
U\cup V$.  This contradicts the fact that $K_\beta\subseteq U\cup V$.
The contradiction proves that $K_\beta$ is connected.

If the connected level set $K_\beta$ did not contain any~$O$-critical
point, we could use the flow generated by the vector field of
Lemma~\ref{l19} to push it out into a level~$K_{r+\epsilon}$
for~$\epsilon>0$, contradicting the minimality of~$\beta$.  Hence
$K_\beta$ must contain a critical point, and thus $r>\beta\geq
\LCD(O)$, proving the theorem.
\end{proof}


\begin{thebibliography}{AI}

\bibitem{Ch91} Cheeger, J.  Critical points of distance functions and
applications to geometry.  Geometric topology: recent developments
(Montecatini Terme, 1990), 1--38, Lecture Notes in Math., 1504,
Springer, Berlin, 1991.

\bibitem{Gr81} Gromov, M.  Curvature, diameter and Betti numbers.
\emph{Comment. Math. Helv.}  \textbf{56} (1981), no.~2, 179--195.

\bibitem{GS} Grove, K.; Shiohama, K.  A generalized sphere theorem.
\emph{ Ann. of Math. (2)} \textbf{106} (1977), no.~2, 201--211.

\bibitem{Ka89} Katz, M.  Diameter-extremal subsets of spheres.
\emph{Discrete and Computational Geometry} \textbf{4} (1989),
117--137.

\end{thebibliography}
\end{document}